\documentclass{amsart}
\usepackage{amssymb}
\usepackage{amscd}

 \newcommand{\QQ}{\mathbf{Q}} 
\newcommand{\ZZ}{\mathbf{Z}} 
 
\newcommand{\1}{{\mathbf{1}}}

\newcommand{\et}{\mathrm{\acute{e}t}}
\newcommand{\id}{\mathrm{id}}

\newcommand{\an}{\mathrm{an}}
 
\newcommand{\antr}{\mathrm{a.t.}}

\newcommand{\eff}{\mathrm{eff}}
\newcommand{\cst}{\mathrm{Artin}}

\DeclareMathOperator{\rk}{rk}

\DeclareMathOperator{\GL}{GL}

\DeclareMathOperator{\Br}{Br}

\DeclareMathOperator{\Gal}{Gal}
\DeclareMathOperator{\Hom}{Hom}
\DeclareMathOperator{\End}{End}

\DeclareMathOperator{\M}{M}

\newcommand{\Han}{H_{\an}}

\newcommand{\tM}{{t\mathcal{M}}}

\newcommand{\IHom}{{\mathcal Hom}}

\newcommand{\Kdagt}{{K^\dagger\mspace{-1.5mu}\{t\}}}
\newcommand{\Kdagtt}{{K^\dagger\mspace{-1mu}(\{t\})}}

\newcommand\isbydef{\stackrel{\text{\tiny def}}{=}}

\newcommand\dash{\nobreakdash-\hspace{0pt}}

\theoremstyle{plain}
\newtheorem{theorem}[subsubsection]{Theorem}
\newtheorem{corollary}[subsubsection]{Corollary}
\newtheorem{lemma}[subsubsection]{Lemma}

\newtheorem{proposition}[subsubsection]{Proposition}
\theoremstyle{definition}
\newtheorem{definition}[subsubsection]{Definition}
\newtheorem{example}[subsubsection]{Example}
\newtheorem{remark}[subsubsection]{Remark}
\newtheorem{remarks}[subsubsection]{Remarks}

\newtheorem{void}[subsubsection]{}

\begin{document}

\markboth{Lenny Taelman}
 {Artin $t$-Motifs}

\title
 {Artin $t$-Motifs}

\author{Lenny Taelman}
\thanks{Supported by a grant of the Netherlands Organisation for Scientific
Research (NWO)}

\begin{abstract}
 We show that analytically trivial $t$\dash motifs satisfy a Tannakian duality, without restrictions
 on the base field, save for that it be of generic characteristic. We show that the group of
 components of the t-motivic Galois group coincides with the absolute Galois group of the base field.
\end{abstract}

\maketitle

\section{Introduction}

Let $k$ be a finite field with $q$ elements and denote by
$k[t]$ the polynomial ring in one variable $t$ over $k$.
Let $K$ be a field containing $k$ and $\theta\in K$ a
distinguished element. Write $\tau$ for the endomorphism
`Frobenius $\otimes$ identity' of $K[t]=K\otimes_k k[t]$,
in other words: $\tau(\sum a_it^i)=\sum a_i^qt^i$.
In \cite{Anderson86} Anderson defined a $t$\dash motive
to be a pair $(M,\sigma)$ consisting of a $K[t]$\dash module
$M$ and a function $\sigma:M\to M$ subject to the conditions that
\begin{enumerate}
 \item[i.] $M$ is free and finitely generated over $K[t]$;
 \item[ii.] $\sigma$ is semi-linear with respect to $\tau$;
          ($\sigma(m_1+m_2)=\sigma(m_1)+\sigma(m_2)$ and
           $\sigma(fm)=\tau(f)\sigma(m)$ for all
           $m,m_1,m_2 \in M$ and $f\in K[t]$);
 \item[iii.] the determinant of $\sigma$ with respect to one
             (equivalently: any) $K[t]$\dash basis of $M$ vanishes
             only at $t=\theta$;
 \item[iv.] there exists a finite set $S\subset M$
           such that $\cup_n \sigma^n(S)$ spans $M$ as a 
           $K$\dash vector space. ($M$ is
           `finitely generated over $K[\sigma]$');
\end{enumerate}

He showed in \emph{loc. cit.} that the category of $t$\dash motives
contains the (opposite) category of Drinfeld modules as a full
subcategory. But there are plenty of interesting
$t$\dash motives that are not Drinfeld modules: for example,
the category of $t$-motives is closed under direct sums
and tensor products while the subcategory of Drinfeld modules
is not closed under either operation.

Recently, Papanikolas has shown \cite{Papanikolas08} that
\emph{analytically trivial} (or \emph{uniformisable})
$t$\dash motifs (\emph{see} \S \ref{thetmotivicgal} for the
definition) satisfy a Tannakian duality, at least when $K$ is
algebraically closed. For the transcendence applications of
\emph{loc. cit.} it is of course sufficient to work over an
algebraically closed field,  but there is also a very arithmetic
flavour to $t$\dash motifs and it is therefore desirable to have a
Tannakian duality over any base field $K$. 

In this paper paper we show that such a duality indeed holds
over general $K$: that a suitable category of $t$\dash motifs over
$K$ is equivalent to the category of representations of some
affine group scheme $\Gamma_K$ over $k(t)$ (the subscript $K$
is present only to denote dependency on $K$.)
(\S \ref{constructingthe} and \S \ref{thetmotivicgal})

We follow \emph{loc. cit.} closely, save in two instances
where we simplify the construction a bit:
\begin{enumerate}
 \item[i.] the closure under internal hom is effected by
           a formal inversion of the `Carlitz twist';
 \item[ii.] there is no condition
            `finitely generated over $K[\sigma]$'
            in the definitions.
\end{enumerate}

The first one is of course completely analogous to
the inversion of the Tate twist in Algebraic Geometry.
Some form of such inversion in the function field context
is already present in \cite{Taguchi95}.

The second one is justified in \S \ref{subsecfingen}, where
we show that any $t$\dash motif (in our terminology) becomes
finitely generated over $K[\sigma]$ after a suitable Tate twist. 
Hence the resulting Tannakian category coincides with the one
generated by those $t$\dash motifs that are finitely
generated over $K[\sigma]$.

Having established a Tannakian duality for $t$\dash motifs
over a general $K$, we show that the group of components
of $\Gamma_K$ is isomorphic with the
absolute Galois group of $K$. (In particular we obtain that
when $K$ is separably closed, the fundamental group $\Gamma_K$
is connected.) This is obtained by a careful analysis
of those $t$\dash motifs that are trivialised by a 
finite separable extension of $K$. This class of $t$\dash motifs
forms the proper analogue of the so-called Artin motifs
in Algebraic Geometry; this is why we refer to them 
as \emph{Artin $t$\dash motifs}. (\S \ref{artintmotifs}
and \S \ref{concomp}) 

\subsection{Notation}
A few words on notation are in order. Let $R$ be a ring and
$\tau:R\to R$ an endomorphism of $R$. An
additive function $\sigma$ between two left $R$\dash modules is said to be semi-linear with respect to $\tau$ if it satisfies the identity
$\sigma(rm)=\tau(r)\sigma(m)$. In this note (almost) all
semi-linear functions are denoted by the Greek letter $\sigma$, and the different endomorphisms according to which they are semi-linear are all denoted by $\tau$. This should not lead to confusion.

\section{Constructing the Category of {\it t}-Motifs}
\label{constructingthe}

\subsection{Effective {\it t}-Motifs}
 Let $k$ be a finite field of $q$ elements and $K$ a field
 containing $k$. Fix a homomorphism $k[t]\to K$  of
 $k$\dash algebras and denote the image of $t$ by $\theta$.
 We do not demand that $k[t]\to K$ be injective now.
 We shall frequently refer to
 `the field $K$', this is silently understood
 to contain the structure homomorphism $k[t]\to K$.
 
 Denote by $\tau$ the endomorphism of $K[t]$ determined by
 $\tau(x)=x^q$ for all $x\in K$ and $\tau(t)=t$.
 The following definition
 goes back to \cite{Anderson86}, although here a 
  slightly less restrictive form is used.
 \begin{definition}
  An \emph{effective $t$-motif} of rank $r$ over $K$ 
  is a pair $M=(M,\sigma)$ consisting of
  \begin{itemize}
   \item a free and finitely generated $K[t]$\dash module $M$
         of rank $r$, and,
   \item a map $\sigma:M\to M$ satisfying
         $\sigma(fm)=\tau(f)\sigma(m)$ for all $f\in K[t]$ and
         $m\in M$,
  \end{itemize}
  such that the determinant of $\sigma$ with respect to some
  (and hence any) $K[t]$\dash basis of $M$ is a power of $t-\theta$
  up to a unit in $K$.
 \end{definition}
 A morphism of effective $t$-motifs is a morphism of $K[t]$\dash
 modules making the obvious square commute. The group of
 morphisms is denoted $\Hom_\sigma(M_1,M_2)$.
 The resulting category of effective $t$\dash motifs over $K$ is
 denoted by $\tM_\eff(K)$ or simply by $\tM_\eff$.

\begin{example}\label{carlitz}
 For any field $k[t]\to K$, the pair 
 \[ C\isbydef K[t] e\ \text{ with }\
    \sigma(fe)\isbydef \tau(f)(t-\theta) e \]
 is an effective $t$\dash motif
 and we will call it the \emph{Carlitz $t$\dash motif}.
 This is the function field counterpart to the
 Lefschetz motif.
\end{example}

\begin{void}
 Define the \emph{tensor product} of two effective $t$\dash
 motifs as
 \[ M_1\otimes M_2 \isbydef M_1\otimes_{K[t]} M_2 \
    \text{ with }\
    \sigma(m_1\otimes m_2)\isbydef\sigma(m_1)\otimes \sigma(m_2). \]
 This is again an effective $t$\dash motif. 
 
 The pair $(K[t],\tau)$  is an
 effective $t$\dash motif which we shall denote
 $\1$. We call it  the \emph{unit $t$\dash motif},
 since for every $M$, one has
 natural isomorphisms $M\otimes \1 = M $ and $\1 \otimes M=M$.
\end{void}

\begin{remark}
 We shall follow a convention sometimes used in representation
 theory and write $nM$ for the direct sum of $n$ copies of $M$
 and $M^n$ for the $n$\dash fold tensor power of $M$.
\end{remark}

 If $M_1$ and $M_2$ are effective $t$\dash motifs then
 $\Hom(M_1,M_2)$ is naturally a $k[t]$\dash module. 
 We have the following \cite[Thm 2]{Anderson86}
\begin{proposition}
  $\Hom(M_1,M_2)$ is free and finitely generated over $k[t]$.
  \qed
\end{proposition}

\subsection{Duality}

\begin{void}
 Let $M_1$ and $M_2$ be effective $t$\dash motifs over $K$.
 Inspired by the theory of linear representations of groups
 we could try to assign to $M_1$ and $M_2$ an effective
 $t$\dash motif of internal homomorphisms as
  \[ \IHom(M_1,M_2)\isbydef\Hom_{K[t]}(M_1,M_2)\
        \text{ with }\
     \sigma(f)\isbydef\sigma\circ f \circ \sigma^{-1}, \]
 where $\sigma\circ f \circ \sigma^{-1}$ is to be read as
 $\sigma_2\circ f \circ \sigma_1^{-1}$.
 This does, however, not make sense, since $\sigma_1$ need 
 not be invertible. First of all, $K$ need not be perfect,
 and secondly---more seriously---the determinant of $\sigma_1$
 is $(t-\theta)^d$ up to a constant, and hence not
 invertible if  $d>0$.
\end{void}

\begin{void}\label{inthom}
 This can be partially resolved. Write $K^a$ for some algebraic
 closure of $K$. Note that after extension of
 scalars from $K[t]$ to $K^a(t)$ the induced action of
 $\sigma$ on $M_1\otimes_{K[t]}K^a(t)$ \emph{is} invertible.
 \begin{proposition}
  For $n$ sufficiently large, the subgroup
  \[\Hom_{K[t]}\hspace{-2pt}\big(M_1,
        M_2\otimes C^{n}\big)\subset
      \Hom_{K^a(t)}\hspace{-2pt}\big(M_1\otimes K^a(t),
                 M_2\otimes C^{n}\otimes K^a(t)\big) \]
  is stable under
   $ f \mapsto \sigma \circ f \circ \sigma^{-1}$. 
 \end{proposition}
 
 \begin{proof}
  Choose bases and express $\sigma$ on $M_1$ and $M_2$ by
  matrices $S_1$ and $S_2$ respectively. 
  Then $M_2\otimes C^n$ has
  a basis on which $\sigma$ is expressed by the
  matrix $(t-\theta)^n S_2$.
  The map $f\mapsto \sigma\circ f\circ \sigma^{-1}$
  translates to a map
   \[ \M(r_2\times r_1,K^a(t))\to
      \M(r_2\times r_1,K^a(t)): F\mapsto F' \]
  with
  \begin{align}
             F' &= (t-\theta)^n S_2
                   \tau(F \tau^{-1}(S_1^{-1}))
                   \notag\\
                &= (t-\theta)^n S_2 \tau(F) S_1^{-1}
                   \notag
  \end{align}
  The Proposition claims that $\M(r_2\times r_1,K[t])$ is
  mapped into itself. But since the determinant of $S_1$ is
  a power of $(t-\theta)$, the matrix $(t-\theta)^n S_1^{-1}$
  has entries in $K[t]$ when $n$ is sufficiently large.
  This immediately implies that $\M(r_2\times r_1,K[t])$ is
  mapped into itself.
  \end{proof}
\end{void}
\begin{void}
  It follows from the proof that
  $\sigma(f)\isbydef \sigma\circ f\circ \sigma^{-1}$
  induces the structure of an effective $t$\dash motif on
  $\Hom_{K[t]}(M_1,M_2\otimes C^n)$ for large $n$.
  We shall denote it by $\IHom(M_1,M_2\otimes C^n)$.
  These internal homs are stable for growing $n$ in
  the sense that there are natural isomorphisms
  \begin{equation}\label{inthomstab}
   \IHom(M_1,M_2\otimes C^n)\otimes C\
      \to\
      \IHom(M_1,M_2\otimes C^{n+1})
  \end{equation}
  relating them.
\end{void}

\subsection{Carlitz Twist and {\it t}-Motifs}

 The previous section hints that the
 obstruction to having internal homs will be lifted
 as soon as the Carlitz $t$\dash motif is made
 invertible. (Very reminiscent of the inversion of the
 Lefschetz motif in the construction of the category of pure
 motifs: If $X$ is a smooth and projective variety of dimension
 $d$ then $\ell$\dash adic Poincar\'{e} duality defines a
 perfect pairing
 \[ H^i_\et(X_{K^s},\QQ_\ell)\times
    H^{2d-i}_\et(X_{K^s},\QQ_\ell)
    \to \QQ_\ell(-d) \]
 which suggests that the motif $h^i(X,\QQ)$ is dual to
 $h^{2d-i}(X,\QQ)$ shifted by the $d$\dash th power
 of the Lefschetz motif, see \cite[\S 4.1]{Andre04}.)
 This can be done quite
 easily, because of the following Lemma, whose verification
 is straightforward.
 \begin{lemma}
  If $M_1$ and $M_2$ are effective $t$\dash motifs, then 
  the natural map
  \[ \Hom_\sigma(M_1,\,M_2)\ \to\
     \Hom_\sigma(M_1\otimes C^{n},\,
                        M_2\otimes C^{n}) \]
  that takes $f$ to $f\otimes \id$ is an isomorphism.\qed
 \end{lemma}

 We are now ready to make the following definition.
 \begin{definition}
  A \emph{$t$\dash motif} is a pair $(M,i)$ consisting of
  an effective $t$\dash motif $M$  and an integer
  $i\in \ZZ$. \emph{Morphisms} between
  $t$\dash motifs are defined by
   \[\Hom_\sigma\hspace{-2pt}\big((M_1,i_1),(M_2,i_2)\big)\
     \isbydef\
  \Hom_{\sigma}\hspace{-2pt}\big(M_1\otimes C^{n+i_1},
                      M_2\otimes C^{n+i_2}\big),\]
  for $n$ sufficiently large.
  The resulting category  is denoted by  $\tM(K)$ or simply by $\tM$.
 \end{definition}
 It suffices to take $n\geq\max(-i_1,-i_2)$ in
 the definition.
 The module of morphisms is independent of $n$ by the preceding
 Lemma.

\begin{void} 
 The functor $M \mapsto (M,0) $
 is fully faithful and we will identify $\tM_\eff$ with
 its image in $\tM$.

 The natural isomorphism between
 $M\otimes C^{n+1}$ and $M\otimes C^{n}\otimes C$
 defines a distinguished isomorphism of $t$\dash motifs
 \begin{equation}\label{stab}
   (M,i+1) = (M\otimes C,i).
 \end{equation}
 In particular, for $i>0$ we can identify
 $C^{i}$ with $(\1, i)$. But note
 that $(\1,i)$ is an object in $\tM$ even when $i$ is
 negative. 
 \end{void}

\begin{void}
 The operations $\oplus$ and $\otimes$ and
 $\IHom$ extend from the category of effective $t$\dash
 motifs---or parts thereof---to the full category of $t$\dash
 motifs:
 \begin{align*}
  \big(M_1,\,i_1\big)\oplus \big(M_2,\,i_2\big) &\isbydef
    \big(M_1\otimes C^{n+i_1}\ \oplus\
     M_2\otimes C^{n+i_2},\,-n\big) \\
  \big(M_1,\,i_1\big)\otimes \big(M_2,\,i_2\big) &\isbydef
    \big(M_1\otimes M_2,\,i_1+i_2\big) \\
  \IHom\big((M_1,i_1),(M_2,i_2)\big) &\isbydef
    \big(\IHom(M_1,M_2\otimes C^{i_1-i_2+n}),\,-n\big)
  \end{align*}
 The occurrences of $n$ in these definitions
 should be read `with $n$ sufficiently large'. Using
 the isomorphisms (\ref{inthomstab}) and (\ref{stab}), one
 verifies that these are independent of $n$ and
 coincide with the operations on effective $t$\dash motifs,
 whenever defined. 
\end{void}

 From now on we will often drop the integer $i$ from the notation and
 write $M$ for a $t$\dash motif, effective or not.
 
\begin{void}
 As usual, we define the dual of a $t$\dash motif $M$ to be
 $M^\vee\isbydef \IHom(M,\1)$. 
 The operations of direct sum, tensor product, duality and
 internal hom satisfy the expected relations---those
 familiar from the theory of linear representations of groups.
 In particular, there is an \emph{adjunction formula}
 \begin{equation}\label{adjunction}
  \Hom(M_1\otimes M_2,M_3)=\Hom(M_1,\IHom(M_2,M_3)).
 \end{equation}
 Also, taking duals is \emph{reflexive}: the natural morphism
 \begin{equation}\label{reflexivity}
  M \to (M^\vee)^\vee
 \end{equation}
 is an isomorphism. And finally, $\IHom$ is \emph{distributive} over
 $\otimes$ in the sense that the natural morphism
 \begin{equation}\label{distributivity}
    \IHom(M_1, M_3)\otimes\IHom(M_2,M_4)\to
    \IHom(M_1\otimes M_2,M_3\otimes M_4),
 \end{equation}
 is an isomorphism. These identities are easily verified. 
\end{void}

 \begin{remark}\label{preabelian}
  The category $\tM$ has kernels and 
  cokernels. This can be seen as follows.
  All morphisms in $\tM$ become morphisms of effective
  $t$\dash motifs after an appropriate shift with a tensor
  power of the Carlitz motif. It is thus sufficient to show
  that $\tM_\eff$ has kernels and cokernels.

  Let $M_1\to M_2$ be a morphism of effective $t$\dash motifs.
  Its group-theoretic kernel is automatically a $t$\dash
  motif and a kernel in the category $\tM_\eff$. The cokernel
  of $f$ in the pre-abelian category of \emph{free} $K[t]$\dash
  modules---the ordinary cokernel modulo torsion---inherits
  an action of $\sigma$ and one verifies that this 
  defines an effective $t$\dash motif
  and a cokernel of $f$ in $\tM_\eff$. Hence $\tM$ is
  pre-abelian. 
  
  Of course $\tM_\eff$ is not abelian, since for example the
  multiplication-by-$t$ map $\1\to \1$ has trivial kernel
  and cokernel, but is not an isomorphism.
\end{remark}

 The existence of kernels and cokernels, and the existence
 of an internal hom bifunctor satisfying 
 (\ref{adjunction}), (\ref{reflexivity}) and
 (\ref{distributivity}) is summarised in:
 \begin{theorem}
   $\tM_K$ is a rigid $k[t]$\dash linear
   pre-abelian tensor category.\qed
 \end{theorem}

\section{The {\it t}-Motivic Galois Group $\Gamma$}
\label{thetmotivicgal}

Passing to the category $\tM^\circ_K$ of objects
`up to isogeny' (\emph{see} \S \ref{subsecisog} below
for the definition) we obtain a rigid $k(t)$\dash linear
abelian tensor category.

On a suitable subcategory of $\tM^\circ_K$ (closed under
$\otimes$, $\oplus$, \ldots) we define
a neutral fibre functor: a fully faithful exact functor to
the category of finite dimensional $k(t)$\dash vector spaces.
The main theorem on Tannakian duality then asserts that this
subcategory is equivalent with the category of $k(t)$\dash linear
representations of some affine group scheme $\Gamma=\Gamma_K$
over $k(t)$. (\emph{see} \S \ref{subsecfibre})

In constructing this subcategory and fibre functor, we follow
closely Papanikolas \cite{Papanikolas08}.

\subsection{Isogenies}\label{subsecisog}

\begin{definition}
 An \emph{isogeny} between two effective $t$\dash motifs
 $M_1$ and $M_2$ is by definition a morphism
  $f\in \Hom_{\sigma}(M_1,M_2)$ such that there
  exists a $g\in \Hom_{\sigma}(M_2,M_1)$
  and a nonzero $h$ in $k[t]$ with $fg=h\,\id=gf$.
\end{definition} 

 The category whose objects are effective $t$\dash motifs
 over $K$ and whose hom-sets are the modules
 $\Hom_{\sigma}(-,-)\otimes_{k[t]}k(t)$ is denoted
 by $\tM^\circ_\eff(K)$. Sometimes we
 will refer to its objects as \emph{effective
 $t$\dash motifs up to isogeny}.

 Denote by $M(t)$ the $K(t)$\dash module
 $M\otimes_{K[t]} K(t)$. The action of $\sigma$ on $M$ extends
 naturally and makes $M(t)$ into a $K(t)[\sigma]$\dash module.
 \begin{proposition}
  The natural map
   \[ \Hom_{\sigma}(M_1,M_2)\otimes_{k[t]} k(t) \
            \to \
      \Hom_{K(t)[\sigma]}(M_1(t),M_2(t)) \]
  is an isomorphism.
 \end{proposition}
 Hence the functor $M \mapsto M(t)$ is fully faithful
 on $\tM^\circ_\eff$. We shall identify $\tM^\circ_\eff$ with
 its image in the category of $K(t)[\sigma]$\dash modules. If we
 take $M_1$ and $M_2$ in the Proposition to be the unit
 $t$\dash motif $\1$, we obtain that the field of
 invariants $K(t)^\sigma$ equals $k(t)$.
 \begin{proof}[Proof of the
Proposition] (\emph{See}
 also \cite{Papanikolas08}.)
 Note that the map is $k(t)$\dash linear. Injectivity is clear.
  
 To show surjectivity, choose $K[t]$\dash
 bases for $M_1$ and $M_2$ and express the action of
 $\sigma$ on them through matrices $S_1$
 and $S_2$. Expressed on the
 induced bases for $M_1(t)$ and $M_2(t)$, a $K(t)[\sigma]$\dash
 homomorphism from $M_1(t)$ to $M_2(t)$ is a matrix $F$ over
 $K(t)$ that satisfies
 \begin{equation}\label{funceq} S_2^{-1} F S_1 = \tau(F).
 \end{equation}
 Let  $h$ be the minimal common denominator of the
 entries of $F$, that is, the minimal
 monic polynomial in $K[t]$ with the property that $h F$ has
 entries in $K[t]$. The minimal common denominator of
 the entries of the right-hand-side $\tau(F)$ is
 $\tau(h)$ and the minimal common denominator of the
 left hand side is $(t-\theta)^rh$ for some $r$.
 Equating them yields $r=0$ and $\tau(h)=h$, hence
 the Proposition.
\end{proof}

 This Proposition has an important consequence:
 \begin{corollary}
  $\tM^\circ_\eff$ is an abelian $k(t)$\dash linear
  tensor category. $\tM^\circ$ is a rigid abelian $k(t)$\dash linear
  tensor category.
 \end{corollary}
Recall that $\tM_\eff$ is not abelian (remark \ref{preabelian}.)
\begin{proof}[Proof of the Corollary]
  The kernels and cokernels in $\tM^\circ_\eff$ are just the
  ordinary group-theoretic kernels and cokernels in 
  the category of left $K(t)[\sigma]$\dash modules,
  and it is clear that a morphism whose kernel and cokernel vanish
  is an isomorphism, and that $\tM^\circ_\eff$ is abelian.
  
  That $\tM^\circ$ is abelian is implied by the abelianness of
  $\tM^\circ_\eff$ and that it is rigid is implied by the
  rigidity of $\tM$, the required properties of $\IHom$ are
  preserved under extension of scalars from $k[t]$ to $k(t)$.
  \end{proof}

\subsection{A Fibre Functor}\label{subsecfibre}

In this section we construct a neutral fibre functor on a
sub-category of $\tM^\circ(K)$, where $k[t]\to K$ is 
assumed to be injective. This construction occurs
already in \cite{Anderson86} and is interpreted as a fibre functor
in \cite{Papanikolas08}. I do not know if there exists a
neutral fibre functor on \emph{all} of $\tM^\circ$.

\begin{void} 
 Let $K^\dagger$ be a field containing $K$ that is algebraically
 closed and complete with respect to a valuation $\|\cdot\|$. 
 Denote by $\Kdagt\subset K^\dagger[[t]]$ the subring
 of restricted power series, that is, those power series whose
 coefficients converge to $0$. In particular, these series have
 a radius of convergence greater than or equal to $1$.
 Note that $\Kdagt$ is closed under $\tau$---raising
 all coefficients to the $q$\dash th power.  A $\tau$\dash
 invariant power series has coefficients in the finite field
 $k$ and hence is restricted if and only if it is a polynomial
 in $t$. That is, $\Kdagt^\tau\!=k[t]$. Denote by
 $\Kdagtt$ the field of fractions of $\Kdagt$. In the next
 paragraph we shall show that $\Kdagtt^\tau\!=k(t)$.
\end{void}

\begin{void} \label{intrat}
 Define the functors $\Han(-,k[t])$ and $\Han(-,k(t))$ on
 the category $\tM_\eff$ of effective $t$\dash motifs as
 \begin{align*}
   \Han(M,k[t]) &\isbydef
     \big(M\otimes_{K[t]} \Kdagt\big)^\sigma\\
   \Han(M,k(t)) &\isbydef 
     \big(M\otimes_{K[t]} \Kdagtt\big)^\sigma
 \end{align*}
  The functors $\Han(-,k[t])$ and $\Han(-,k(t))$ are related:
 \begin{proposition}
  $\Han(M,k(t))=\Han(M,k[t])\otimes k(t)$.
 \end{proposition}
 Taking $M$ to be $\1$ yields 
 \begin{corollary}
  $\Kdagtt^\tau\!=k(t)$.\qed
 \end{corollary}
 \begin{proof}[Proof of the Proposition]
  The ring $\Kdagt$ is a principal
  ideal domain and every non-zero ideal is of the form 
   \[ (t-\alpha_1)(t-\alpha_2)\cdots (t-\alpha_n)\Kdagt
     \subset \Kdagt \]
  with $\|\alpha_i\|\leq 1$ for all $i$
  \cite{Tate71}.
  
  Take a $g\in\Han(M,k(t))$ and write it 
  as $h^{-1}m$ with $m\in M\otimes \Kdagt$ and 
  $h$ a finite product $h=\prod(t-\alpha_i)$. Assume that
  the degree of $h$ is minimal.
  The invariance of $g=h^{-1}m$ gives
   \[ \tau(h)m =h\sigma(m) \in M\otimes\Kdagt,\]
  hence the minimality of $h$ implies $\tau(h)=h$.
 \end{proof}
\end{void}

 \begin{proposition}
  If $\|\theta\|>1$  then
   $ \Han(C,k[t])\approx k[t]$ and $\Han(C,k(t))\approx k(t)$.
 \end{proposition}

 \begin{proof} (Cf. Lemma 2.5.4 of
 \cite{Anderson90}.)
  Consider the product expansion
   \[ \Omega \isbydef \frac{1}{\sqrt[q-1]{-\theta}}
           \prod_{i\geq 0}\left(1-\frac{t}{\theta^{q^i}}\right), \]
  where $\sqrt[q-1]{-\theta}$ is any root in $K^\dagger$. (Any two
  such roots differ by a scalar in $k^\times$.)
  The infinite product converges for all values of $t$
  and all zeroes have absolute value greater than or equal
  to $\|\theta\| > 1$, thus $\Omega\in K\{t\}^\times$. By
  construction
   $\Omega = (t-\theta) \tau(\Omega)  $
  and therefore $\Han(C,k[t])) = k[t]\Omega e$.
  Similarly $\Han(C,k(t)) = k(t)\Omega e$.
 \end{proof}
 Henceforth, when considering the functors $\Han$, we shall
 always assume that $\|\theta\| > 1$.
 
\begin{void}\label{generalfibre}
 So far, we have only considered \emph{effective} $t$\dash
 motifs. Shifting back and forth with powers of the Carlitz
 motif, we can extend the functors $\Han$ to functors defined
 on all $t$\dash motifs as follows
  \begin{align*}
   \Han((M,i),k[t]) &\isbydef
     \Han(M,k[t])\otimes_{k[t]} \Han(C,k[t])^{\otimes i}\\
   \Han((M,i),k(t)) &\isbydef
     \Han(M,k(t))\otimes_{k(t)} \Han(C,k(t))^{\otimes i}.
 \end{align*}
 The resulting functor is well-defined by the
 canonical isomorphisms
 \[ \Han(M\otimes C,k[t])=\Han(M,k[t])\otimes_{k[t]}\Han(C,k[t]). \]
\end{void}

\begin{void}  
 These functors are not faithful. In fact, we shall
 shortly see that there exist non-trivial $M$ with
 $\Han(M,k[t])=0$.
 \begin{definition}
  A $t$\dash motif $(M,i)$ over $K$ is said to
  be \emph{analytically trivial} if one of the following equivalent
  conditions holds:
  \begin{itemize}
   \item $M\otimes_{K[t]} \Kdagt$ has a
      $\sigma$\dash invariant $\Kdagt$\dash basis,
   \item $M\otimes_{K[t]} \Kdagtt$ has a
      $\sigma$\dash invariant $\Kdagtt$\dash basis,
   \item $\rk_{k[t]} \Han(M,k[t]) = \rk M$,
   \item $\dim_{k(t)} \Han(M,k(t)) = \rk M$.
  \end{itemize}
  Denote by $\tM_\antr\subset \tM$ and
  $\tM^\circ_\antr\subset \tM^\circ$ the full subcategories
  consisting of the analytically trivial objects.
 \end{definition}
 The analytic triviality of a $t$\dash motif $M$
 depends on the embedding of $K$ into $K^\dagger$---see
 \ref{nonunif} for an example. When 
 we are dealing with the category $\tM_\antr$ we will
 assume that such an embedding has been fixed.
 
 \begin{proof}[Proof of equivalence]
  By paragraph \ref{intrat} the first condition is equivalent
  with the second, and the third with the fourth. Clearly 
  the first implies the third.  To conclude the converse
  (that the third implies the first), it
  suffices to show that for all effective
  $t$\dash motifs $M$ the natural map
   \[ \big(M\otimes \Kdagt\big)^\sigma
    \otimes_{k[t]} \Kdagt \to
      M\otimes \Kdagt \]
  is injective. This can be done exactly as in
  \cite[Thm 2]{Anderson86},
  using  $\Kdagt^\tau=k[t]$.
 \end{proof}
\end{void}

 Some immediate consequences of the definition are:
 \begin{theorem}\label{hanfaithful}
 The class of analytically trivial $t$\dash motifs
 is closed under tensor product and duality. Moreover
  \begin{itemize}
   \item $\Han(-,k[t])$ is a faithful $k[t]$\dash
         linear $\otimes$\dash functor on $\tM_\antr$,	 
   \item $\Han(-,k(t))$ is an exact, faithful, $k(t)$\dash
         linear $\otimes$\dash functor on $\tM^\circ_\antr$.	 
  \end{itemize}
  In particular $\tM^\circ_\antr(K)$ is neutral
  Tannakian with fibre functor $\Han(-,k(t))$.\qed
 \end{theorem}
 
 It follows from the main Theorem on neutral Tannakian
 categories \cite[Thm 2.11]{Deligne82a} that 
 $\tM^\circ_\antr(K)$ is equivalent with the category
 of $k(t)$\dash linear representations of some
 affine group scheme $\Gamma_K$ over $k(t)$.
 
 But note that $\Gamma_K$ depends on the chosen valuation on
 $K$. We shall usually tacitly assume that a valuation (with
 $\|\theta\|>1$) has been fixed.

 We will end this subsection with an example of 
 a $t$\dash motif that is not analytically trivial.

\begin{example}\label{nonunif}
 Not all $t$\dash motifs are analytically trivial.
 Consider for example, the rank $2$ 
 effective $t$\dash motif
 \[
    M_\xi=K[t] e_1+K[t] e_2\ \text{ with }\ 
    \begin{cases} \sigma(e_1)=\xi t e_1 + e_2\\
                  \sigma(e_2)=e_1 \end{cases} \]
 depending on a parameter $\xi\in K$.

 \emph{Claim:} $M_\xi$ is analytically trivial
  if and only if $\|\xi\|<1$.

 In particular there exists no valuation on $K$ for which
 $M_\xi$ is analytically trivial when $\xi$ is algebraic over
 $k$.
\end{example}

 \begin{proof}[Proof of the Claim]
  Assume that 
  $\Kdagt e_1+\Kdagt e_2$ has
  an invariant vector $ae_1 + b e_2$, with $a,b$
  in $\Kdagt$. Expressing the invariance under $\sigma$ gives
  \[ \begin{cases}a=\tau(a)\xi t + \tau^2(a)\\
                  b=\tau(a)\end{cases} \]
  Expand $a=a_0+a_1t+\cdots$ with $a_i\in K^\dagger$.
  Then it follows that
  $a_0^{q^2}=a_0$, that is, $a_0$ lies in the quadratic extension
  $l/k$ inside $K^\dagger$, and in particular $\|a_0\|=1$
  (assuming $a_0\neq 0$). The higher
  $a_i$ satisfy the recurrence equation 
   \begin{equation}\label{recur} a_n-a_n^{q^2} = \xi a_{n-1}^q.
   \end{equation}
  If $\|\xi\|\geq 1$ then $\|a_n\|\geq 1$ for all $n$ and the
  series $a_0+a_1t+\cdots$ is therefore not a restricted series,
  confirming one direction of the Proposition. If on the other hand
  $\|\xi\|<1$ then define an $a$ recursively by taking at every
  step the unique solution $a_n$ of (\ref{recur}) that
  has $\|a_n\|< 1$. This produces a restricted power series
  for every $a_0\in l^\times$ and it suffices
  to take two $a_0$'s independent over $k$ to obtain two independent
  invariant vectors for $M_\xi\otimes_{K[t]}\Kdagt$.
  \end{proof}

\section{Artin {\it t}-Motifs}\label{artintmotifs}

\subsection{Definition}

 A corollary to Lang's Theorem \cite{Lang56} asserting the
 surjectivity of the map 
  \[ \GL(n,K^s)\to \GL(n,K^s) : A \mapsto \tau(A)A^{-1} \]
 is \cite[Prop. 4.1]{Pink06}:
 \begin{theorem}
 The category of pairs $(V,\sigma)$,
 consisting of a finite dimensional
 $K$-vector space $V$ and an additive map $\sigma:V\to V$
 satisfying $\sigma(x v)=x^q v$ and such that 
 $K\sigma(V)=V$ is neutral Tannakian $k$\dash linear
 with fibre functor
  \[ V \mapsto (V\otimes_K K^s)^\sigma \]
 and fundamental group $G_K\isbydef \Gal(K^s/K)$.\qed
\end{theorem}

\begin{remarks}
The invariants $(-)^\sigma$ are taken for the diagonal
action $V\otimes_K K^s$ induced by the given action
of $\sigma$ on $V$ and the $q$\dash th power map
on $K^s$. This is the unique  extension of
$\sigma:V\to V$ to a semi-linear map 
$V\otimes K^s\to V\otimes K^s$.

If $K$ is not perfect $K\sigma(V)$ need not coincide
with $\sigma(V)$, as can be seen already when
$(V,\sigma)=(K,x\mapsto x^q)$.
 
 We abusively write $G_K$ for both the pro-finite group and
 the corresponding constant affine group scheme over $k$
(obtained
 as the limit of the system of finite constant group schemes
 corresponding to the finite quotients of the pro-finite
 group.) Their categories of
 representations on finite dimensional $k$\dash vector spaces
 coincide.
\end{remarks}
 
 A pair $(V,\sigma)$ induces an effective $t$\dash motif
 $M(V)\isbydef V\otimes_K K[t]$ where the action of $\sigma$
 is induced from the action on $V$.

 We would like to interpret the collection of $t$\dash
 motifs $M(V)$ as a Tannakian subcategory of
 $\tM^\circ$, but there are of course many more
 morphisms $M(V_1)\to M(V_2)$ 
 than morphisms $V_1\to V_2$ and the kernel and cokernel
 of a morphism from $M(V_1)$ to $M(V_2)$ are typically not of the
 form $M(V)$. 
 \begin{proposition}\label{cstmot}
  Let  $M$ be an effective $t$\dash motif over $K$.
  The following are equivalent:
  \begin{itemize}
   \item $M$ is isomorphic to a sub-quotient of $M(V)$ for some $V$,
   \item $M\otimes_K {K^s}\approx n\1$ for some $n$.
  \end{itemize}
  \end{proposition}

  \begin{definition}
    An \emph{Artin $t$\dash motif} is an effective $t$\dash
    motif $M$ satisfying the above equivalent conditions.
  \end{definition}
  
  \begin{proof}[Proof of the Proposition]
   If $M$ is a sub-quotient of $M(V)$ then $M_{K^s}$ is a sub-quotient
   of $M(V_{K^s})\approx m\1$ and therefore $M_{K^s}\approx n\1$.
   
   Conversely, assume that $M_{K^s}$ has a basis of
   $\sigma$\dash invariant vectors.
   There exists some finite extension $K'/K$ inside $K^s$ such
   that this basis is already defined over $K'$.
   The natural map $K[t]\to K'[t]$
   defines the structure of a $K[t]$\dash module on $M'$. Denote
   it by $R_{K'/K} M'$ in order to distinguish it from the
   $K'[t]$\dash module $M'$. It is clear that
   $R_{K'/K}M'$ is naturally an effective $t$\dash motif over
   $K$ of rank $\rk(M)[K':K]$. (Call it
   the \emph{Weil restriction} of $M'$ from $K'$ to $K$.)
   But, $M$ is a submodule of $R_{K'/K}M'\otimes_K{K^s}$
   and the latter is isomorphic to $M(R_{K'/K}W)$ with
   $W$ the sum of a number of copies of $K'$ with the diagonal
   action of $\sigma$, whence the Proposition.
  \end{proof}

 The full subcategory $\tM^\circ_\cst(K)$ of $\tM^\circ(K)$
 consisting of the Artin $t$\dash motifs 
 is rigid abelian $k(t)$\dash linear and has a fibre functor
 \begin{equation}\label{constantfibre}
   M \rightsquigarrow 
  \big(M\otimes_{K[t]} K^s[t]\big)^\sigma \otimes_{k[t]} k(t)
 \end{equation}
 and with this fibre functor we have
 \begin{proposition}\label{consttannaka}
  $\tM^\circ_\cst(K)$ is neutral Tannakian $k(t)$\dash linear
  with fundamental group $G_K$.
 \end{proposition} 
 
 Note that it is not needed to use analytic methods to obtain
 a fibre functor on Artin $t$\dash motifs and in particular
 it is not needed to demand that $k[t]\to K$ be injective.
 
 \begin{proof}[Proof of the Proposition]
  The functor $M(V)\rightsquigarrow H(V)\otimes_k k(t)$
  induces a fully faithful embedding of $\tM^\circ_\cst(K)$ into
  the category of $k(t)$\dash linear representations of $G_K$.
  It will be essentially surjective as soon as every
  continuous $k(t)$\dash linear representation of $G_K$ is a
  sub-quotient of $H\otimes_k k(t)$ for some
  $k$\dash linear representation $H$. This is indeed so, since
  every (algebraic, or continuous) representation of $G_K$
  factors though a finite group $G$ and every
  representation of $G$ is a sub-quotient of the direct sum
  of a number of copies of the regular representation
  $k(t)[G]$, which is nothing but the regular representation
  $k[G]$ over $k$, tensored with $k(t)$.
 \end{proof}

\begin{void}
 Artin $t$\dash motifs are the $t$\dash counterparts of
 the algebro-geometric Artin motifs.
 Let $\ZZ\to K$ be any field.
 Consider the category of smooth and projective varieties $X$
 over $K$ that are of dimension zero. These are the spectra of
 the finite \'{e}tale $K$\dash algebras and by
 Grothendieck's
 formulation of Galois theory the category of such $X$ is
 equivalent to the category of finite $G_K$\dash
 sets. The motifs that are sub-quotients of the
 $h(X,\QQ)$ for zero-dimensional $X$ are called
 Artin  motifs. They form a category which is equivalent to
 the category of $\QQ$\dash linear
 representations of $G_K$, see \cite[\S 4.1]{Andre04}.
\end{void}

\subsection{Relation between $\Gamma_K$ and $\Gamma_{K^s}$}

 Suppose now that $k[t]\to K$ is actually injective. Choose
 $K^\dagger\supset K$ to be algebraically closed,  complete and with
 $\|\theta\|>1$. Let $K^s$ be the separable closure of $K$
 inside $K^\dagger$. For an Artin $t$\dash motif $M$
 we have that 
 \[ \big(M \otimes_{K[t]} K^s[t]\big)^\sigma \otimes_{k[t]} k(t) =
    \big(M \otimes_{K[t]} \Kdagtt\big)^\sigma. \]
 That is to say, $\tM(K)_\cst^\circ$
 is a full sub-category of $\tM(K)_\antr^\circ$ and the analytic
 fibre functor on the latter extends the algebraic
 fibre functor on the former.
 \begin{theorem}\label{arithgeom}
  There is a short exact sequence
  \[ 0 \to \Gamma_{K^s} \to \Gamma_K \to G_K \to 0 \]
  of affine group schemes over $k(t)$.
 \end{theorem}

\begin{proof}
 The full subcategory $\tM_\cst^\circ(K)$ of $\tM_\antr^\circ(K)$
 is Tannakian with fundamental group $G_K$ (\ref{consttannaka})
 and is closed under
 sub-quotients in $\tM_\antr^\circ$ by definition.
 This implies the existence of a faithfully flat, and hence
 surjective, morphism $\Gamma_K\to G_K$ of affine group
 schemes. (see for example
 \cite[Prop. 2.21 (a)]{Deligne82a}.)

 If $M$ is an effective $t$\dash motif
   over $K^s$, then it has a model $M'$ over a finite
   extension $K'$ of $K$. The $t$\dash motif $M$ is a submotif of
   $R_{K'/K}M'\otimes_K{K^s}$. Thus
   every $t$\dash motif over $K^s$ is a
   submotif of a $t$\dash motif that is already
   defined over $K$. It follows that the
   fully faithful functor $M\rightsquigarrow M_{K^s}$ from
   $\tM^\circ_\antr(K)$ to $\tM^\circ_\antr(K^s)$ defines
   a closed immersion $\Gamma_{K^s}\to \Gamma_K$.
   (see \cite[Prop. 2.21 (b)]{Deligne82a})

  The sequence is exact in the middle if and only if the
   representations of $\Gamma_K$ on which $\Gamma_{K^s}$ acts
   trivially are precisely those coming from a representation
   of $G_K$. In other words, the exactness is equivalent with
   the statement that a $t$\dash motif $M$ over $K$ satisfies
   $M_{K^s}\approx n\1$ for some $n$ if and only if it is an
   Artin $t$\dash motif. This was one of the equivalent
   definitions of the notion of an Artin $t$\dash motif
   (see \ref{cstmot}).
\end{proof}

\section{Weights}

\subsection{Dieudonn\'{e} {\it t}-Modules}

 As usual $k$ is a finite field of $q$ elements and
 $K$ a field containing $k$. Denote by $\tau$ the continuous
 endomorphism of the field of Laurent series $K((t^{-1}))$ that
 fixes $t^{-1}$ and that restricts to the $q$\dash th power 
 map on $K$.
 \begin{definition}
  A \emph{Dieudonn\'{e} $t$\dash module}\footnotemark
  \ over $K$ is a pair $(V,\sigma)$ of
  \begin{itemize}
   \item a finite-dimensional $K((t^{-1}))$\dash vector space $V$
         and
   \item an additive map $\sigma:V\to V$ satisfying
         $\sigma(fv)=\tau(f)\sigma(v)$ for all $f\in K((t^{-1}))$
         and all $v\in V$, 
  \end{itemize}
  such that $K\sigma(V)=V$.
 \end{definition}
 A morphism of Dieudonn\'{e} $t$\dash modules is of course a
 $K((t^{-1}))$\dash linear map commuting with $\sigma$.

 Dieudonn\'{e} $t$\dash modules are easily classified, at least over
 a separably closed field. The main `building blocks' are the
 following modules:
 \begin{definition}\label{dieuclass}
  Let $\lambda=s/r$ be a rational number with
  $(r,s)=1$ and $r>0$. The 
  Dieudonn\'{e} $t$\dash module $V_\lambda$ is defined to be the
  pair $(V_\lambda,\sigma)$ with 
  \begin{itemize}
   \item $V_\lambda\isbydef K((t^{-1}))e_1\oplus\ldots\oplus
          K((t^{-1}))e_r$
   \item $\sigma(e_i)\isbydef e_{i+1}$ ($i<r$) and
	  $\sigma(e_r)\isbydef t^se_1$
  \end{itemize}
 \end{definition}
 The classification states:
 \begin{proposition}
  If $V$ is a  Dieudonn\'{e} $t$\dash module over a separably closed
  field $K$ then there exist rational numbers
  $\lambda_1,\cdots,\lambda_n$ such that
  \begin{itemize}
   \item $V \approx V_{\lambda_1}\oplus\cdots\oplus
                    V_{\lambda_n}$, and,
   \item the $t^{-1}$\dash adic valuations of the roots of
         the characteristic polynomial of $\sigma$ expressed
         on \emph{any} $K((t^{-1}))$\dash basis are
         $\{-\lambda_i\}_i$, each counted with multiplicity
         $\dim V_{\lambda_i}$.
  \end{itemize}
  If $\lambda\neq\mu$ then $\Hom(V_\lambda,V_\mu)=0$.
  For all $\lambda$, the ring $\End(V_{\lambda})$ is a
  division algebra over $k((t^{-1}))$. Its Brauer class is
  $\lambda+\ZZ\in\QQ/\ZZ=\Br(k((t^{-1})))$.
 \end{proposition}
 Note that this classification is formally identical to
 the classification of the classical ($p$\dash adic) Dieudonn\'{e}
 modules \cite{Dieudonne57}.
 \begin{proof}
  This is shown in \cite[Appendix B]{Laumon96}. Although
  the statements therein are made only for a particular field
  $K$, nowhere do the proofs make use of anything stronger
  then the separably closedness of $K$.
 \end{proof}

 The following characterisation of $V_\lambda$ is useful.
 \begin{proposition}\label{monoweight}
  Let $V$ be a Dieudonn\'{e} $t$\dash module over a separably
  closed field $K$ and $\lambda$ a rational number.
  The following are equivalent:
  \begin{itemize}
    \item $V\approx V_\lambda\oplus V_\lambda\oplus
               \cdots \oplus V_\lambda$;
    \item there exists a lattice $\Lambda\subset V$ such
          that $\sigma^r(\Lambda)=t^s\Lambda$ where $r$ and $s$
          are coprime integers with $\lambda=s/r$.
  \end{itemize}
 \end{proposition}
 \begin{proof}
  \emph{One $\Rightarrow$ Two.} If $V=V_\lambda$ and
  $(e_i)$ the basis that occurs in its definition
  (\ref{dieuclass}) then the lattice generated by the same
  basis $(e_i)$ has the required property. For
  $V=V_\lambda\oplus\cdots\oplus V_\lambda$ it thus suffices
  to take the lattice $\Lambda\oplus\cdots\oplus\Lambda$.
  
  \emph{Two $\Rightarrow$ One.} The operator $t^{-s}\sigma^r$
  transforms a $K[[t^{-1}]]$\dash basis of $\Lambda$ into a
  new $K[[t^{-1}]]$\dash basis of $\Lambda$ and
  therefore has eigenvalues of valuation $0$. 
  \end{proof}

\subsection{Weights}

\begin{void}
 Let $K$ be separably closed. Let $M$ be an effective
 $t$\dash motif over $K$. Then
  \[ M((t^{-1}))\isbydef M \otimes_{K[t]} K((t^{-1})) =
                         M(t) \otimes_{K(t)} K((t^{-1})) \]
 is a Dieudonn\'{e} $t$\dash module. The displayed equality
 shows that it only depends on the isogeny class of $M$.
 By the classification of Dieudonn\'{e} $t$\dash
 modules (\ref{dieuclass})
 there exist rational numbers $\lambda_1,\ldots,\lambda_n$
 such that
  \[ M((t^{-1})) \approx V_{\lambda_1} \oplus\cdots\oplus
                         V_{\lambda_n}. \]
 We call these rational numbers the \emph{weights} of $M$. 
 If $K$ is not separably closed then we define the weights
 of an effective $t$\dash motif $M$ to be the weights of
 $M_{K^s}$. This clearly does not depend on the choice of a
 separable closure.
 
 We say that $M$ is \emph{pure of weight $\lambda$} if
 the only weight occurring is $\lambda$. By Proposition
 \ref{monoweight}, this  coincides with the definition as given
 in \cite{Anderson86}.
\end{void} 

 We now collect a number of facts related to the notions
 of weights and purity. They are either immediate consequences of
 the definitions or well-known facts established in the literature.
 \begin{proposition}\label{pureprop}
  We have the following:
  \begin{itemize}
   \item If $M$ is pure of weight $\lambda$ then every
         sub-quotient of $M$ is pure of weight $\lambda$;
   \item If $M$ has a filtration in which all successive quotients
         are pure of weight $\lambda$, then $M$ is pure of weight
         $\lambda$;
   \item If the sets of weights of $M_1$ and $M_2$ are disjoint then
         $\Hom(M_1,M_2)=0$;
   \item $C$ is pure of weight $1$;
   \item The weights of $M_1\otimes M_2$ are the sums of weights
         of $M_1$ with those of $M_2$;
   \item The weight of a pure effective $t$\dash motif $M$ is
         non-negative.
  \end{itemize}
 \end{proposition}

 \begin{proof}[Proofs]
  \emph{One.} If $M'$ is a sub-quotient of $M$ then
  $M'((t^{-1}))$ is a sub-quotient of $M((t^{-1}))$ and the claimed
  statement follows at once from the Classification \ref{dieuclass}.
  
  \emph{Two.} A normal series of $M$ induces a normal series of
  $M((t^{-1}))$ and again the contention follows from
  \ref{dieuclass}.
  
  \emph{Three.} $\Hom(M_1,M_2)$ is a submodule of
  $\Hom(M_1((t^{-1})),M_2((t^{-1})))$, which is zero by
  \ref{dieuclass}.
    
  \emph{Four.}  The valuation of $t-\theta$ at $t^{-1}$ is $-1$.

  \emph{Five.} Immediate since the zeroes of the characteristic
   polynomials are multiplied.
   
  \emph{Six.} Clear for rank one $M$, for a general $M$ take
  the top exterior power.
 \end{proof}

\begin{void}
 If $M$ is an effective $t$\dash motif and
  \[ M((t^{-1}))\approx V_{\lambda_1} \oplus\cdots\oplus
                        V_{\lambda_n} \]
 then by the Proposition
  \[ (M\otimes C)((t^{-1})) \approx
                     V_{\lambda_1+1} \oplus\cdots\oplus
                     V_{\lambda_n+1}. \]
 It is thus natural to define the \emph{weights} of a
 $t$\dash motif $(M,i)$ to be the set of $\lambda+i$
 where $\lambda$ runs through the weights of $M$. To be
 consistent, a $t$\dash motif $(M,i)$ is then said to be
 \emph{pure of weight $\lambda$} if and only if $M$ is
 pure of weight $\lambda-i$.
\end{void}

\subsection{Finite Generation over $K[\sigma]$}\label{subsecfingen}

Let $K$ be algebraically closed. 
\begin{theorem}
 If $M$ is an effective $t$\dash motif and all weights 
 of $M$ are positive then $M$ is finitely generated
 as a $K[\sigma]$\dash module.
\end{theorem}
\begin{corollary}
 If $M$ is an effective $t$\dash motif then for all
 $n$ sufficiently large $M\otimes C^n$ is finitely generated
 over $K[\sigma]$.\qed
\end{corollary}
\begin{remark}
 It follows in particular that the analytically trivial
 effective $t$\dash motifs that are finitely generated 
 over $K[\sigma]$ generate the Tannakian category
 $\tM^\circ_\antr$, and that the $t$\dash motivic Galois
 groups constructed here coincide with those of
 \cite{Papanikolas08}. 
\end{remark}
\begin{proof}[Proof of the Theorem]
 (\emph{Cf.} \cite[Prop. 1.9.2]{Anderson86})
 There is an isomorphism
  \[ M((t^{-1})) \approx V_{u_1/v} \oplus \cdots 
                         \oplus V_{u_k/v}. \]
 Let $\Lambda$ be the $K[[t^{-1}]]$\dash lattice
 in $M((t^{-1}))$ that corresponds with the standard lattice
 in the right hand side (see \ref{monoweight}).
 Let $u$ be the minimum of the $u_i$.
 By the hypothesis $u>0$. Then
  \[ t^u \Lambda \subset \sigma^v \Lambda. \]

 Define an increasing filtration 
  \[ M_0 \subset M_1 \subset M_2 \subset \cdots \]
 by 
  \[ M_n \isbydef M \cap t^{nu}\Lambda, \]
 the intersection taken inside $M((t^{-1}))$. Clearly
 $M=\cup_n M_n$.
 
 \emph{Claim:} $M_{n+1} \subset M_n + \sigma^v M_n$ for all
 sufficiently big $n$.
 
 Since the $M_n$ are of finite dimension over $K$, the claim
 implies at once that $M$ is finitely generated over $K[\sigma]$.
 
 To prove the claim, note that for all $n$ sufficiently large
 \begin{equation}\label{sufflarge} 
   M + t^n\Lambda = M((t^{-1})).
 \end{equation}
 For such $n$ we have
 \begin{eqnarray*}
   M_{n+1} &=&     M \cap t^{nu}t^u\Lambda \\
       &\subset&  M \cap t^{nu}\sigma^v\Lambda \\
       &\subset& (M \cap t^{nu})
           +(\sigma^vM \cap \sigma^vt^{nu}\Lambda)\\
       &=& M_n + \sigma^v M_n
\end{eqnarray*}
 where the second inclusion needs (\ref{sufflarge}).
\end{proof}

\section{Connected Components of $\Gamma$}\label{concomp}

The morphism $k[t]\to K$ is assumed to be injective.

\subsection{The Tate Conjecture}

Let $\lambda$ be a monic irreducible element of $k[t]$. Denote
by $k(t)_\lambda$ the $\lambda$-adic completion of $k(t)$.
To an effective $t$\dash motif $M=(M,\sigma)$ over $K$ we
associate a
$\lambda$\dash adic Galois representation as follows:
 \[ M \mapsto H_\lambda(M) :=
     \varprojlim_n (M_{K^s}/\lambda^n M_{K^s})^\sigma
      \otimes_{k[t]} k(t). \]
Here $(-)^\sigma$ denotes invariants for the induced action of
$\sigma$ on $M_{K_s}/\lambda^n M_{K^s}$ (note that $\sigma$
and $\lambda^n$  commute.) 
It follows from Lang's Theorem \cite{Lang56}
that $H_\lambda(M)$ is a vector space over $k(t)_\lambda$
whose dimension equals the rank of $M$. By transport of 
structure $G_K$ acts continuously on $H_\lambda(M)$.

The construction extends to give a $k(t)$\dash linear
$\otimes$\dash functor
$H_\lambda$ from $\tM^\circ$ to the category of finite dimensional
continuous $\lambda$\dash adic $G_K$\dash representations.

Let us now restrict attention to analytically trivial $t$\dash motifs.
Fix a $\lambda$ as above.
On the category $\tM^\circ_\antr$ we now have two fibre functors
to $k(t)_\lambda$\dash vector spaces: 
$H_\an(-,k(t)_\lambda)$, (defined as $H_\an(-,k(t))\otimes k(t)_\lambda$) and $H_\lambda(-)$. By the formalism of Tannakian categories
there exists an isomorphism $\alpha_\lambda$ of fibre functors
  \[ \alpha_\lambda(-) : H_\an(-,k(t)_\lambda)\to H_\lambda(-). \]
Also it follows that if $M$ is a $t$\dash motif with
Tannakian fundamental group $G$, then the image of the 
$\lambda$\dash adic Galois representation is contained in
$G(k(t)_\lambda)$, via the identification $\alpha_\lambda(M)$.
  
We have the following fundamental result (\cite{Taguchi95},
\cite{Tamagawa94}; see also \cite[\S 19]{Stalder07}):
\begin{theorem}[``Tate conjecture'']
 Assume $K$ is finitely generated and let $\lambda$
 be a monic irreducible element of $k[t]$, coprime with the
 kernel of $k[t]\to K$. Then for all
 $M_1,M_2\in\tM^\circ(K)$ the natural homomorphism
  \[ \Hom(M_1,M_2)\otimes_{k(t)} k(t)_\lambda \to
     \Hom_{k(t)_\lambda[G_K]}(H_\lambda(M_1),H_\lambda(M_2)) \]
 is an isomorphism.\qed
\end{theorem}

\subsection{Connected Components of $\Gamma$}

 \begin{theorem}
  $\Gamma_{K^s}$ has no finite quotients. In particular
  it is connected.
 \end{theorem}
 \begin{proof}
  Note that $\Gamma\to\pi_0(\Gamma)$ is a pro-finite
  \'{e}tale quotient, hence the second statement indeed follows
  from the first. 
  
  Let $G$ be a finite quotient of $\Gamma_{K^s}$. 
  To
  a faithful representation (say, the regular representation)
  of $G$ corresponds an analytically trivial
  $t$\dash motif $M$ over $K^s$. It suffices to show that $M$
  is constant, for then it is trivial (over $K^s$) and consequently
  $G$ is trivial.
  
  Now, $M$ is defined over some finitely generated $L\subset K^s$, 
  hence the map $\Gamma_{K^s}\to G$ factors as
   \[  \Gamma_{K^s} \to \Gamma_L \to G. \]
  Since $G(k_\lambda)$ is finite, the $\lambda$\dash adic
  representation of $G_L$ associated with
  $M_L$ becomes trivial over some finite extension $L'/L$ inside
  $K^s$. Thus by the Tate conjecture, $M_{L'}$ is trivial,
  hence $M$ was constant.
\end{proof}

\bibliographystyle{plain}
\bibliography{../../master}

\end{document}